\documentclass[preprint,12pt]{elsarticle}

\usepackage{amssymb,amsmath,amsopn}
\usepackage{amsthm}
\usepackage{color}

\newcommand{\remark}[1][]{\noindent \emph{Remark. }}
\newcommand{\example}[1][]{\noindent \emph{Example. }}
\newcommand{\note}[1][]{\noindent \emph{Note. }}
\newtheorem{theorem}{Theorem}

\newtheorem{defi}{Definition}
\newtheorem{problem}{Problem}
\newtheorem{conj}{Conjecture}
\newtheorem{corollary}{Corollary}

\DeclareMathOperator{\lin}{lin}

\journal{Linear Algebra and its Applications}

\begin{document}

\begin{frontmatter}

\title{Isometries of Minkowski geometries.}

\author{\'Akos G.Horv\'ath}

\address{\'A. G.Horv\'ath, Dept. of Geometry, Budapest University of Technology,
Egry J\'ozsef u. 1., Budapest, Hungary, 1111}

\ead{ghorvath@math.bme.hu}

\begin{abstract}
In this paper we review the known facts on isometries of  Minkowski geometries and prove some new results on them. We give the normal forms of two special classes of operators and also characterize the isometry group of Minkowski $3$-spaces in which the unit sphere does not contain an ellipse.
\end{abstract}


\end{frontmatter}

\section{Introduction}
The one hundred year old concept of ``Minkowski space" is a nice topic of recent geometric research. Nevertheless, the phrase ``Minkowski space" is applied for two different theories: the theory of normed linear spaces and the theory of linear spaces with indefinite metric. It is interesting (see \cite{gho 1},\cite{gho 2},\cite{gho 3}) that these essentially distinct theories have similar axiomatic foundations. The axiomatic build-up of the theory of linear spaces with indefinite metric comes from H. Minkowski \cite{minkowski} and the similar system of axioms of normed linear spaces was introduced by Lumer much later in \cite{lumer}. The first concept widely used in physics is the mathematical structure of relativity theory and thus its importance is without doubt. On the other hand, the importance of the second theory is based on the fact that a large part of modern functional analysis works in so-called normed spaces which are more general ones than inner product (or Hilbert) spaces.  This motivates the introduction of the so-called semi-inner product which is an important tool of the corresponding investigations. Of course, in both of these two theories a lot of problems can be formulated or can be solved in the language of geometry. Our theme of interest is the theory of finite-dimensional, separable and real semi-inner product spaces. Such a normed space with the branches of its geometric properties is called \emph{Minkowski geometry}.

Our purpose is to review the possible characterizations of the distinct transformation groups of Minkowski geometry, take into consideration the analytic theory and also the synthetic geometric-algebraic investigations. Through the paper we prove some new statements. We mention Theorem \ref{nfaa} and Theorem \ref{nfi} which introduce  normal forms for the adjoint abelian operators and isometries of a Minkowski $n$-space. Theorem \ref{igtne} describes the isometry group of a Minkowski $3$-space with the property that its unit sphere does not contain an ellipse. This latter result generalizes a theorem of H.Martini, M. Spirova and K. Strambach proved for non-Euclidean Minkowski planes.

\section{Operator theory of Minkowski geometry}

A generalization of inner product and inner product spaces was raised by G. Lumer in \cite{lumer}.
\begin{defi}[\cite{lumer}]
The semi-inner product (s.i.p) on a complex vector space $V$ is a complex function $[x,y]:V\times V\longrightarrow \mathbb{C}$ with the following properties:
\begin{description}
\item[s1]: $[x+y,z]=[x,z]+[y,z]$,
\item[s2]: $[\lambda x,y]=\lambda[x,y]$ \mbox{ for every } $\lambda \in \mathbb{C}$,
\item[s3]: $[x,x]>0$ \mbox{ when } $x\not =0$,
\item[s4]: $|[x,y]|^2\leq [x,x][y,y]$,
\end{description}
A vector space $V$ with a s.i.p. is an s.i.p. space.
\end{defi}

G. Lumer proved that an s.i.p space is a normed vector space with norm $\|x\|=\sqrt{[x,x]}$ and, on the other hand, that every normed vector space can be represented as an s.i.p. space. In \cite{giles} J. R. Giles showed that all normed vector spaces can be represented as s.i.p. spaces with homogeneous second variable. Giles also introduced the concept of \emph{ continuous s.i.p. space} as an s.i.p. space having the additional property:
For any unit vectors $x,y \in S$, $\Re\{[y,x+\lambda y]\}\rightarrow\Re\{[y,x]\}$ for all real $\lambda\rightarrow 0$.
Giles proved in \cite{giles} that an s.i.p. space is a continuous (uniformly continuous) s.i.p. space if and only if the norm is G\^{a}teaux (uniformly Fr\`{e}chet) differentiable. Giles also proved a variation of the Riesz representation theorem.

\subsection{Self-adjoint operators and the generalized adjoint}

Without using the concept of semi-inner product, on the "self-adjoint" property of a linear operator of a real normed space can be said on the basis of the concept of \emph{norm derivative} (see in \cite{alsina}, \cite{dragomir} or in \cite{wojcik}).
If $(V,\|\cdot\|)$ is a real normed space, the functions
$$
\rho_{\pm}'(x,y):=\lim\limits_{t\to \pm 0}\frac{\|x+ty\|^2-\|x\|^2}{2t}=\|x\|\cdot\lim\limits_{t\to \pm 0}\frac{\|x+ty\|-\|x\|}{t}
$$
are called \emph{norm derivatives}. Their properties are
\begin{itemize}
\item $\forall x,y\in V \quad \forall \alpha\in \mathbb{R}  \quad \rho_{\pm}'(x,\alpha x+y)=\alpha \|x\|^2+\rho_{\pm}'(x,y)$;
\item $\forall x,y\in V \quad \forall \alpha\in \mathbb{R}^+ \quad \rho_{\pm}'(\alpha x,y)=\alpha \rho_{\pm}'(x,y)=\rho_{\pm}'(x,\alpha y)$;
\item $\forall x,y\in V \quad \forall \alpha\in \mathbb{R}^- \quad \rho_{\pm}'(\alpha x,y)=\alpha \rho_{\mp}'(x,y)=\rho_{\pm}'(x,\alpha y)$;
\item $\forall x\in V \quad \rho_{\pm}'(\alpha x,x)=\|x\|^2$;
\item $\forall x,y\in V \quad |\rho_{\pm}'(x,y)|\leq \|x\|\|y\|$;
\item $\forall x,y\in V \quad \rho_{-}'(x,y)\leq \rho_{+}'(x,y)$;
\item $\forall x,y,z\in V \quad \rho_{+}'(x,y+z)\leq \rho_{+}'(x,y)+\rho_{+}'(x,z)$;
\item $\forall x,y,z\in V \quad \rho_{-}'(x,y+z)\geq \rho_{-}'(x,y)+\rho_{-}'(x,z)$;
\item $\rho_{\pm}'(x,y)$ continuous with respect to the second variable, but not necessarily with respect to its first one;
\item if $(V,[\cdot,\cdot])$ is a smooth semi inner product space then $[y,x]=\rho_{+}'(x,y)=\rho_{-}'(x,y)$ (implying also that $\rho_{+}'(x,y)$ is linear in its second argument.
\end{itemize}
Thus the concept of norm derivatives does not differ from the concept of the s.i.p. in smooth spaces. In a non-smooth space the set of smooth points is denoted by $D_{sm}(V)$. From Mazur's theorem (see \cite{mazur}) immediately follows that in a separable real Banach space $(V,\|\cdot\|)$ the set $D_{sm}(V)$ is dense. In this context, \emph{self-adjoint operator} means a linear operator $A$ that satisfies the property $\forall x,y\in V \quad \rho_{+}'(A(x),y)=\rho_{+}'(x,A(y))$.
In \cite{wojcik} we can see the following theorem:
\begin{theorem}[\cite{wojcik}]
Let $(V,\|\cdot\|)$ be a separable real Banach space. The mapping $f:V\longrightarrow V$ satisfies the functional equation
$$
\forall x,y\in V \quad \rho_{+}'(f(x),y)=\rho_{+}'(x,f(y))
$$
if and only if $f$ is a self-adjoint operator of $V$. (Hence it is linear and continuous.)
\end{theorem}
The case of the Banach space $C(M)$ of the continuous functions of a compact metric space $M$ has a high importance in analysis. (The norm is the supremum norm defined by the equality $\|x\|_\infty=\sup\{|x(t)| : t\in M\}$.) For this infinite-dimensional space the following has been proved:
\begin{theorem}[\cite{wojcik}]
Let $T:C(M)\longrightarrow C(M)$ be a non-vanishing mapping. Then the following statements are equivalent:
\begin{enumerate}
\item[(1)] $T$ satisfies the functional equation $\rho_{+}'(T(x),y)=\rho_{+}'(x,T(y))$ for all $x,y\in V$;
\item[(2)] there exist a scalar $\gamma\neq 0$ and a linear isometry $U:C(M)\longrightarrow C(M)$ such that $U^2=Id$ and $T=\gamma U$;
\item[(3)] there exist a scalar $\gamma\neq 0$ and a homeomorphism $h:M\longrightarrow M$ such that $h\circ h= id$ and a continuous function $\sigma:M\longrightarrow \mathbb{R}$ such that $|\sigma(t)|=1$ for $t\in M$ and $T\varphi=\gamma \cdot \sigma \cdot \varphi \circ h$ for all $\varphi\in C(M)$.
\end{enumerate}
\end{theorem}
Statements (2) and (3) for adjoint abelian operators of any normed space have been proved by Fleming and Jamison in \cite{fleming-jamison}.

Turning to Minkowski geometry we recall that a real, finite-dimensional normed space $(V,\|\cdot\|)$ is a \emph{normed space of $S$-type} if there is a basis such that a linear operator $A:V\longrightarrow V$ is self-adjoint if and only if the matrix of $A$ with respect to this basis is symmetric. W\'ojcik proved that if $V$ is a real, finite-dimensional normed space then it is of $S$-type if and only if the norm comes from an inner product.

Let $V$ be a smooth, uniformly convex Banach space with a s.i.p.. If $A$ is a bounded linear operator from $X$ to itself then $g_y(x)=[A(x),y]$, is a continuous linear functional, and from the generalized Riesz-Fischer representation theorem it follows that there is a unique vector $A^T(y)$ such that
$$
[A(x),y]=[x,A^T(y)] \mbox{ for all } x\in X.
$$
$A^T$ is called the \emph{generalized adjoint} of $A$. This mapping is the usual Hilbert space adjoint if the space is an inner product space. In this general set-up this map is not usually linear but it still has some interesting properties. Koehler denoted by $\varphi$ the duality map from $V$ to its dual space $V^*$ given by $\varphi(y)=f_y=[\cdot,y]$. Then $\varphi$ is a duality map, it is one-to-one and onto, and its inverse is always continuous (since $X$ is uniformly convex). $\varphi$  is continuous at a point $x\in V$ if and only if the norm of $V$ is Fr$\acute{e}$chet differentiable at that point \cite{giles 2}.
\begin{theorem}[\cite{koehler}]
For linear transformations $A$ and $B$ and for the scalar $\lambda$ we have:
\begin{enumerate}
\item $(\lambda A)^T=\overline{\lambda}A^T$ \mbox{ and } $(AB)^T=B^TA^T$.
\item $A^T$ is one-to-one if and only if the range of $A$ is dense in $X$.
\item $A^*\varphi=\varphi A^T$.
\item If norm of $V$ is Fr$\acute{e}$chet differentiable then $A^T$ is continuous.
\end{enumerate}
\end{theorem}

\subsection{Characterization of adjoint abelian operators in Minkowski geometry}

Stampfli in \cite{stampfli} has defined a bounded linear operator $A$ to be \emph{adjoint abelian} if and only if there is a duality map $\varphi$ such that $A^*\varphi=\varphi A$. So evidently, $A$ is adjoint abelian if and only if $A=A^T$, thus the adjoint abelian operators are in some sense "self-adjoint" ones.
L\'angi in \cite{langi} introduced the concept of the \emph{Lipschitz property}  of a semi-inner product and investigated the diagonalizable operators of a Minkowski geometry $\{V,\|\cdot\|\}$. He said that the semi-inner product $[\cdot,\cdot]$ has the Lipschitz property if for every $x$ from the unit ball there is e real number $\kappa$ such that for every $y$ and $z$ from the unit ball holds $|[x,y]-[x,z]|\leq \kappa\|y-z\|$. We note that the differentiability property for the semi-inner product (defined in \cite{gho 1}) implies the Lipschitz property of the product, too. Let $A$ be a diagonalizable linear operator of $V$, and let $\lambda_1>\lambda_2>\ldots \lambda_k\geq 0$ be the absolute values of the eigenvalues of $A$. If $\lambda_i$ is an eigenvalue of $A$, then $E_i$ denotes the eigenspace of $A$ belonging to $\lambda_i$, and if $\lambda_i$ is not an eigenvalue, set $E_i=\{0\}$. $E_i$ defined similarly with $-\lambda_i$ in place of $\lambda_i$. The main result in \cite{langi} is the following.
\begin{theorem}[\cite{langi}]
Let $V$ be a smooth finite-dimensional real Banach space such that the induced semi-inner product $[\cdot,\cdot]$ satisfies the Lipschitz condition, and let $A:V\longrightarrow V$ be a diagonalizable linear operator. Then $A$ is adjoint abelian with respect to $[\cdot,\cdot]$ if, and only if, the following hold.
\begin{enumerate}\label{langithm}
\item $[\cdot,\cdot]$ is the direct sum of its restrictions to $\overline{E}_i=\lin\{E_i\cup E_{-i}\}$, $i=1,\ldots, k$;
\item for every value of $i$, the subspaces $E_i$ and $E_{-i}$ are both transversal and normal (meaning that they are mutually orthogonal in the sense of Birkhoff orthogonality);
\item for every value of $i$, the restriction of $A$ to  $\overline{E}_i$ is the product of $\lambda_i$ and an isometry of $\overline{E}_i$.
\end{enumerate}
\end{theorem}
Using an observation from \cite{gho 1} and Corollary 3 from \cite{langi}, we get that -- by the assumption of the theorem -- if no section of the unit sphere with a plane is an ellipse with the origin as its centre, then every diagonalizable adjoint abelian operator of $X$ is a scalar multiple of an isometry of $V$. This motivates the following definition:
\begin{defi}\label{tnedef}
A Minkowski $n$-space is \emph{totally non-Euclidean} if it has no $2$-dimensional Euclidean subspace.
\end{defi}
Now the corollary above says:
\begin{corollary}
In a totally non-Euclidean Minkowski $n$-space every diagonalizable  adjoint abelian operator is a scalar multiple of an isometry.
\end{corollary}

The following theorem describes the structure of a real adjoint abelian operator.

\begin{theorem}\label{nfaa}
Let $V$ be a smooth finite-dimensional real Banach space with the induced semi inner product $[\cdot,\cdot]$. If $A$ is adjoint abelian with respect to $[\cdot,\cdot]$ then $V$ can be decomposed into the direct sum of $A$-invariant subspaces of dimension at most two. Restricting $A$ to a $2$-dimensional component it is a \emph{generalized dilatation} defined by the matrix
$$
\left[A|_{\lin\{a_s,b_s\}}\right]_{\{a_s,b_s\}}=|\lambda|\left(\begin{array}{cc}
\cos\varphi & \sin\varphi \\
-\sin\varphi & \cos\varphi
\end{array}
\right) \mbox{ where } |\lambda|\in\mathbb{R}^+ \mbox{ and } 0< \varphi \leq 2\pi
$$
and the basis $\{a_s,b_s\}$ holds the equalities $[a_s,a_s]=[b_s,b_s]=1$, $[a_s,b_s]=[b_s,a_s]=0$.
\end{theorem}

\begin{proof} First we prove that if $A$ is an adjoint abelian operator and $U$ is an $A$-invariant subspace then the orthogonal complement
$$
U^\bot:=\{v\in V \quad | \quad [v,u]=0 \mbox{ for all } u\in U\}
$$
is also $A$ invariant. In fact,  for a $v\in U^\bot$ we have $[A(v),u]=[v,A(u)]=0$ for all $u\in U$ proving this statement. From this it follows a decomposition of the space $V$ to the direct sum of minimal invariant subspaces $V_i$ with the property $V_i^{\bot}\supset V_j$ for all $j>i$. From the fundamental theorem of algebra it also follows that the dimension of $V_i$ is at most $2$.

Assume that $Z$ is a 2-dimensional minimal invariant subspace of $A:_\mathbb{R}V\longrightarrow _\mathbb{R}V$ implying that it does not contain a real eigenvector of $A$. Hence for every vector $z\in Z$ the pair of vectors $z$ and $A(z)$ form a  basis in $Z$. Thus the equality $A^2(z)=\gamma z + \delta A(z)$ also holds. Since this equation also valid if we substitute the variable vector $z$ into $A(z)$ we get that the polynomial equation $A^2=\gamma I+ \delta A$ holds on $Z$. Set $\delta=2\alpha$, then we get the equation $\left(A-\alpha I\right)^{2}=(\alpha^2+\delta)I$. Since there is no real eigenvalue of $A$ on $Z$ we get that $(\alpha^2+\delta)<0$, say $-\beta^2$. Thus we have that a polynomial equation of second order of form $\left(A-\alpha I\right)^{2}=-\beta^2 I$ is valid on $Z$.

Let $ _\mathbb{C}{Z}$ be the two dimensional complex vector space on the vectors of the additive commutative group $Z$, defined by the set of linear combinations
$$
\left\{\xi f_1+\zeta f_2 \quad \{f_1,f_2\} \mbox{ is a basis of } _{\mathbb{R}}Z \mbox{ and } \xi,\zeta \in \mathbb{C}\right\}
$$
We can decompose the minimal polynomial $\left(x-\alpha\right)^2+\beta^2$ to linear terms by the identity $\left(x-\alpha\right)^2+\beta^2=\left(x-\alpha-\beta i\right)\left(x-\alpha+\beta i\right)$. Hence we can correspond two complex eigenvalues $\lambda=\alpha+ \beta i$ and $\overline{\lambda}=\alpha-\beta i$ of the extracted complex linear operator $\widetilde{A}: _\mathbb{C}Z\longrightarrow _\mathbb{C}{Z}$. (Note that with respect to the basis
$\{f_1,f_2\}$ the complex operator $\widetilde{A}$ has the same (and real) coefficients as of the real linear operator $A$.) In $_\mathbb{C}{Z}$ for the eigenvalues $\lambda $ and $\overline{\lambda}$ have distinct eigenspaces of dimension $1$. These complex lines are generated by the complex vectors
$$
u=\xi f_1+ \zeta f_2=(\alpha_1+\beta_1i)f_1+(\alpha_2+\beta_2i)f_1=
$$
$$
=\left(\alpha_1f_1+\alpha_2f_2\right)+\left(\beta_1f_1+\beta_2f_2\right)i=:a+bi,
$$
and its conjugate $\overline{u}=a-bi$, respectively. (Here $a,b\in _\mathbb{R}Z$.) We say in this case that $\lambda$ is a \emph{complex eigenvalue of the real linear operator} $A$ with \emph{complex eigenvector} $u$. We identify the one-dimensional complex eigenspace of $u$ with the two dimensional real subspace generated by $a$ and $b$ with the mapping $E:_\mathbb{C}<u>\longrightarrow _\mathbb{R}{Z} $
$$
E((x +y i)(a+bi) ):=\mathfrak{R}( (x +y i)(a+bi))+\mathfrak{I}((x +y i)(a+bi))=(x +y)a+(x-y)b.
$$
We note that $E$ is a bijective mapping. In fact, if $x+y=x'+y'$ and $x-y=x'-y'$ then $x=x'$ and $y=y'$ and there is an unique solution of the equation system $r=x+y$ and $s=x-y$, which is $x=(r+s)/2$, $y=(r-s)/2$. From this it follows that we can assume that $a$ and $b$ is an Auerbach basis of $Z$ meaning in the rest part of this proof that $[a,a]=[b,b]=1$ and $[a,b]=[b,a]=0$.

Let now a complex eigenvalue of $A$ be $\lambda$. Denote by $E$ the complex eigenspace (of dimension $d$) corresponding to $\lambda$. Then $\overline{\lambda}$ is an eigenvalue with the eigenspace $\overline{E}$, where $\overline{E}=\{ \overline{u} \quad u\in E\}$.

If $\{u_1,\ldots, u_{d}\}$ is a complex basis of $E$ then $\{\overline{u_1},\ldots, \overline{u_d}\}$ is a basis of $\overline{E}$. Assuming that $u_s=a_s+b_si$ and $\lambda=\alpha+\beta i$, we get that $\overline{u_s}=a_s-b_si$ and $\overline{\lambda}=\alpha-\beta i$. Since
$$
A(a_s)+A(b_s)i=A(u_s)=\lambda u_s=\left(\alpha a_s- \beta b_s\right)+\left(\beta a_s+\alpha b_s\right)i,
$$
$A$ is invariant on the real subspace $\widetilde{E}:=\lin\{a_s,b_s \quad s=1,2,\ldots {d}\}$ which we call the real invariant subspace associated to $\lambda$. It is clear that to the eigenspace $\overline{E}$ we can associate the same invariant subspace. Since the vectors $u_s=a_s+b_si \quad s=1,\ldots ,{d}$ form a basis of the complex subspace $E$, the vectors $\{a_s,b_s \quad s=1,\ldots, {d}\}$ form a real generator system of $\widetilde{E}$ implying that the dimension is at most $2d$. Consider a pair of real vectors $a_s,b_s$. If $b_s=\lambda a_s$ then
$$
a_s(\alpha-\lambda\beta)+a_s(\beta+\alpha\lambda)i= \left(a_s\alpha-b_s\beta\right)+\left(a_s\beta+\alpha b_s\right)i=
$$
$$
=A(a_s+b_si)=(1+i\lambda)A(a_s)=(1+i\lambda )a_s(\alpha-\lambda\beta)=
$$
$$
=a_s(\alpha-\lambda\beta)+i\lambda a_s(\alpha-\lambda\beta),
$$
implying that
$$
\beta+\alpha\lambda=\lambda\alpha-\lambda^2\beta.
$$
Since $\lambda\neq 0$ it follows that $\beta=0$, which contradicts the fact that $\lambda$ is not a real number. This shows that every pair $\{a_s,b_s\}$  contains independent vectors. Thus the complex eigenspace of dimension $d$ is isomorphic to that real space of dimension $2d$ which is the direct product of its two dimensional subspaces generated by $a_s$ and $b_s$.

Hence the adjoint abelian operator $A$ invariant on the real plane $\lin\{a_s,b_s\}$ and with respect to the basis $\{a_s,b_s\}$ has the matrix representation:
$$
A=\left(\begin{array}{cc}
\alpha_{r} & \beta_{r} \\
-\beta_{r} & \alpha_{r}
\end{array}
\right)=|\lambda|\left(\begin{array}{cc}
\cos\varphi & \sin\varphi \\
-\sin\varphi & \cos\varphi
\end{array}
\right)=:|\lambda| F_{\varphi}.
$$
where $|\cdot|$ means the absolute value of a complex number and $\varphi$ is the argument of $\lambda$.
\end{proof}

We note that $F_{\varphi}$ is also an adjoint abelian operator on that plane, we call it \emph{generalized rotation} with respect to the basis $\{a_s,b_s\}$.  In fact, $|\lambda| \neq 0$ because $\lambda$ is not real. Thus we have
$$
\left[F_{\varphi}(x),y\right]=\frac{1}{|\lambda|}\left[|\lambda|F_{\varphi}(x),y\right]=\frac{1}{|\lambda|}\left[x,|\lambda| F_{\varphi}(y)\right]=\left[x,F_{\varphi}(y)\right].
$$
\begin{example} To get a generalized rotation, consider an inner product plane defined by the unit circle
$$
\left(\frac{x}{a}\right)^2+\left(\frac{y}{b}\right)^2=1.
$$
The product is
$$
[v,z]=[x_1e+y_1f,x_2e+y_2f]=\frac{x_1x_2}{a^2}+\frac{y_1y_2}{b^2},
$$
and a required basis is $\{ae,bf\}$. The generalized rotation in the Euclidean orthonormal basis $\{e,f\}$ is
$$
F_{\varphi}=\left(\begin{array}{cc}
\frac{1}{a} & 0 \\
0 & \frac{1}{b}
\end{array}
\right)\left(\begin{array}{cc}
\cos\varphi & \sin\varphi \\
-\sin\varphi & \cos\varphi
\end{array}
\right)\left(\begin{array}{cc}
{a} & 0 \\
0 & {b}
\end{array}
\right)=\left(\begin{array}{cc}
\cos\varphi & \frac{b}{a}\sin\varphi \\
-\frac{a}{b}\sin\varphi & \cos\varphi
\end{array}
\right).
$$
It is an isometry because it sends the unit disk into itself, however it is not an adjoint abelian operator because of
$$
[F_{\varphi}(e),f]=-\frac{a}{b^3}\sin\varphi\neq \frac{b}{a^3}\sin\varphi=[e,F_{\varphi}(f)].
$$
\end{example}

We suspect the following:

\begin{conj}
From Theorem \ref{langithm} (or Theorem 1 (and Corollary 2) in \cite{langi}) we can omit the assumption "diagonalizable". More precisely every adjoint-abelian operator of a smooth Minkowski space is diagonalizable.
\end{conj}

In the case of $l_p$ spaces this conjecture is true:

\begin{theorem}\label{lp}
Let $1<p<\infty$ be a real number. In a finite-dimensional real $l_p$ space every adjoint abelian operator is diagonalizable.
\end{theorem}

\begin{proof}
Observe that for an $l_2$ space the statement is true because the semi inner product is an inner product.
Consider the Euclidean plane with the $l_p$ norm $1<p<\infty$. The corresponding semi inner product (see in \cite{giles}) can be defined by the equality
$$
[z,v]=[x_1a_s+y_1b_s,x_2a_s+y_2b_s]=\frac{1}{\|s_2\|^{p-2}_{p}}\int_X s_1|s_2|^{p-1}\mathrm{sgn }(s_2) d\mu=
$$
$$
=\frac{1}{\left(|x_2|^p+|y_2|^p\right)^{\frac{p-2}{p}}}\left(x_1|x_2|^{p-1}\mathrm{sgn }(x_2)+y_1|y_2|^{p-1}\mathrm{sgn }(y_2)\right),
$$
where $\{a_s, b_s\}$ is an orthonormal basis in the Euclidean sense and Auerbach basis with respect to the $l_p$ norm associated to the above product. Now we have the formulas
$$
[F_\varphi(z),v]=\frac{1}{\left(|x_2|^p+|y_2|^p\right)^{\frac{p-2}{p}}}\left((\cos\varphi x_1+\sin\varphi y_1)|x_2|^{p-1}\mathrm{sgn }(x_2)+\right.
$$
$$
\left. +(\cos\varphi y_1-\sin\varphi x_1)|y_2|^{p-1}\mathrm{sgn }(y_2)\right),
$$
and
$$
[z, F_{\varphi}(v)]=\frac{1}{\left(|\cos \varphi x_2+\sin \varphi y_2|^p+|\cos  \varphi y_2-\sin  \varphi x_2|^p\right)^{\frac{p-2}{p}}} \times
$$
$$
\times \left(x_1|\cos \varphi x_2+\sin  \varphi y_2|^{p-1}\mathrm{sgn }(\cos \varphi x_2+\sin  \varphi y_2)+\right.
$$
$$
\left.+y_1|\cos  \varphi y_2-\sin \varphi x_2|^{p-1}\mathrm{sgn}(\cos  \varphi y_2-\sin  \varphi x_2)\right).
$$
For $\varphi =\pi/2$ we get that
$$
[F_\varphi(z),v]=\frac{1}{\left(|x_2|^p+|y_2|^p\right)^{\frac{p-2}{p}}}\left( y_1|x_2|^{p-1}\mathrm{sgn }(x_2)+\right.
$$
$$
\left.+(-x_1)|y_2|^{p-1}\mathrm{sgn }(y_2)\right)=-[z,F_\varphi(v)]
$$
holds for all $z$ and $v$. Since $[F_\varphi(z),v]=[z,F_\varphi(v)]$ also holds for all $z$ and $v$, we get that $F_\varphi(z)=0$ for all $z$ giving a contradiction. Thus $\varphi\neq \pi/2$ for an adjoint abelian generalized rotation.

If $\varphi=\pi$ then $F_\varphi(v)=-v$ and it is diagonalizable for all $p$.

Finally if $\varphi=3\pi/2$ then
$$
[F_\varphi(z),v]=\frac{1}{\left(|x_2|^p+|y_2|^p\right)^{\frac{p-2}{p}}}\left( y_1|x_2|^{p-1}\mathrm{sgn }(x_2)+
x_1|y_2|^{p-1}\mathrm{sgn }(y_2)\right)
$$
and
$$
[z,F_\varphi(v)]=\frac{1}{\left(|x_2|^p+|y_2|^p\right)^{\frac{p-2}{p}}}\left( y_1|x_2|^{p-1}\mathrm{sgn }(x_2)-
x_1|y_2|^{p-1}\mathrm{sgn }(y_2)\right)
$$
providing the strict inequality $[F_\varphi(z),v]>[z,F_\varphi(v)]$ for $z$ and $v$ with positive $x_1$ and $y_2$. This is a contradiction, too.

For general (and fixed) $\varphi$ we get the equality
$$
{\left(|\cos \varphi x_2+\sin \varphi y_2|^p+|\cos  \varphi y_2-\sin  \varphi x_2|^p\right)^{\frac{p-2}{p}}}\times
$$
$$\left((\cos\varphi x_1+\sin\varphi y_1)|x_2|^{p-1}\mathrm{sgn }(x_2)+(\cos\varphi y_1-\sin\varphi x_1)|y_2|^{p-1}\mathrm{sgn }(y_2)\right)=
$$
$$
={\left(|x_2|^p+|y_2|^p\right)^{\frac{p-2}{p}}}\left(x_1|\cos \varphi x_2+\sin  \varphi y_2|^{p-1}\mathrm{sgn }(\cos \varphi x_2+\sin  \varphi y_2)+\right.
$$
$$
\left.+y_1|\cos  \varphi y_2-\sin \varphi x_2|^{p-1}\mathrm{sgn}(\cos  \varphi y_2-\sin  \varphi x_2)\right),
$$
which holds for all $z$ and $v$.

First we substitute $x_2=y_2$ and $y_1=0$ into this equality and we get:
$$
|x_2|^{2p-3}\left(|\cos\varphi +\sin\varphi |^p+|\cos\varphi -\sin\varphi |^p\right)^{\frac{p-2}{p}}x_1\mathrm{sgn}(x_2)(\cos\varphi -\sin\varphi)=
$$
$$
=|x_2|^{2p-3}|\cos\varphi +\sin\varphi |^{p-1}x_1\mathrm{sgn}(x_2)\mathrm{sgn}(\cos\varphi +\sin\varphi),
$$
implying the other equality
$$
\left(|\cos\varphi +\sin\varphi |^p+|\cos\varphi -\sin\varphi |^p\right)^{\frac{p-2}{p}}(\cos\varphi -\sin\varphi)=
$$
$$
=|\cos\varphi +\sin\varphi |^{p-1}\mathrm{sgn}(\cos\varphi +\sin\varphi).
$$
From this it immediately follows that either $\cos\varphi \pm \sin\varphi >0$ or $\cos\varphi \pm \sin\varphi<0$.

We can also substitute the equalities $y_2=0$ and $x_1=y_1$ into the original equality. This leads to the equality:
$$
\left(|\cos\varphi|^p+|-\sin\varphi |^p\right)^{\frac{p-2}{p}}(\cos\varphi +\sin\varphi)=
$$
$$
=|\cos\varphi|^{p-1}\mathrm{sgn}(\cos\varphi)+|-\sin\varphi|^{p-1}\mathrm{sgn}(-\sin\varphi).
$$
Now from the assumption $\cos\varphi \pm \sin\varphi >0$ it follows that  $\mathrm{sgn}(\cos\varphi)=1$ and we have two possibilities. If $\mathrm{sgn}(-\sin\varphi)=-1$ then we get
$$
\left(1+(\tan\varphi) ^p\right)^{\frac{p-2}{p}}(1 +\tan\varphi)=1-(\tan\varphi)^{p-1}.
$$
Let
$$
f(p):=\left(1+(\tan\varphi) ^p\right)^{\frac{p-2}{p}}(1 +\tan\varphi)-1+(\tan\varphi)^{p-1}
$$
be a function of $p$ for a fixed admissible $\varphi$. It is clear that $\lim_{p\to \infty}f(p)=\tan\varphi$ and a short calculation shows that for $p>2$ it is a non-increasing function which at $p=2$ is $2\tan \varphi$ hence for $p\geq 2$ we get that $f(p)>0$. The function $f(p)$ on the interval $1<p<2$ is concave showing that $f(p)\geq \min\{f(1),f(2)\}>0$. Thus there is no $p$ and $\varphi$ for which this equality can be hold.

If $\mathrm{sgn}(-\sin\varphi)=1$ then we get the equality
$$
\left(1+|\tan\varphi| ^p\right)^{\frac{p-2}{p}}(1-|\tan\varphi|)=1+|\tan\varphi|^{p-1},
$$
and the function
$$
f(p):=1+|\tan\varphi|^{p-1}-\left(1+|\tan\varphi| ^p\right)^{\frac{p-2}{p}}(1 -|\tan\varphi|)>
$$
$$
>1+|\tan\varphi|^{p-1}-1-|\tan\varphi| ^p
$$
is positive for $1<p<\infty$, since $|\tan\varphi|<1$.

Thus there remains only one possibility which could give a non-trivial adjoint abelian generalized rotation in an $l_p$ space (for certain $p$) when we assume that
$\cos\varphi \pm \sin\varphi<0$. In this case $\mathrm{sgn}(\cos\varphi)=-1$ and $|\cos \varphi|>|\sin\varphi|$. However in this case the substitution $y_2=0$ and $x_1=y_1$ leads to the same equalities as in the previous one leading to the same contradictions. Thus there is no non-diagonalizable adjoint abelian generalized rotation in a real $l_p$ space of finite dimension, as we stated.
\end{proof}

We note that in the case of a Minkowski geometry we got a new proof for the fact that every adjoint abelian operator on $L_p$ ($1<p<\infty, \quad p\neq 2$ ) is a multiple of an isometry (see in \cite{fleming-jamison}).

\subsection{Characterization of isometries in Minkowski geometry}

A Banach space isometry is a linear mapping which preserves the norm of the vectors. As it can be seen easily, the following theorem holds.
\begin{theorem}[\cite{koehler}]
A mapping in a smooth Banach space is an isometry if and only if it preserves the (unique) s.i.p..
\end{theorem}
Thus, if the norm is at least smooth, then the two types of linear isometry coincide. On the basis of the results of Stampfli \cite{stampfli}
we have two corollaries:

\begin{corollary}[\cite{koehler}]
In any smooth uniformly convex Banach space, $U$ is an invertible isometry if and only if $U^{-1}=U^T$. As a result if in addition $U^{-1}=U$ then $U$ is scalar.
\end{corollary}

Stampfli has defined an operator $U$ to be \emph{iso-abelian} if and only if there is a duality map $\phi$ such that $\phi U=(U^*)^{-1}\varphi$.

\begin{corollary}[\cite{koehler}]
In a smooth Banach space $U$ is iso-abelian if and only if it is an invertible isometry.
\end{corollary}

The above statement was extended to include the non-smooth case in \cite{koehler-rosenthal}. Precisely:

\begin{theorem}[\cite{koehler-rosenthal}]
Let $V$ be a normed linear space (real or complex) and $U$ be an operator mapping $V$ into itself. Then $U$ is an isometry if and only if there is a semi-inner product $[\cdot,\cdot]$, such that $[U(x),U(y)]=[x,y]$ for all $x$ and $y$.
\end{theorem}

As a corollary of this theorem the authors also proved the following:

\begin{corollary}[\cite{koehler-rosenthal}]
$U$ is iso-abelian if and only if it is an invertible isometry.
\end{corollary}

For our characterization the following result is very important:

\begin{theorem}[\cite{koehler-rosenthal}]
A finite dimensional eigenspace of an isometry has a complement invariant under the isometry.
\end{theorem}

For the construction it can be seen that this complement is orthogonal to the given eigenspace of the isometry with respect to that semi-inner product which is preserved by the isometry. Since every linear mapping has at least one (complex) eigenvalue, hence a complex finite-dimensional Banach space is an orthogonal direct sum of eigenspaces of a given isometry (See Corollary 4 in \cite{koehler-rosenthal}.) For the real case we get analogously the following statement:

\begin{theorem}\label{nfi}
Let $V$ be a finite dimensional real Banach space, $U:V\longrightarrow V$ be an isometry on $V$, and $[\cdot,\cdot]$ be a semi-inner product preserved by $U$. Then there is a decomposition of the space of form
$$
V=V_1\oplus\ldots V_s\oplus V_{s+1}\oplus  \ldots \oplus V_l\oplus V_{l+1}\oplus\ldots \oplus V_{l+k},
$$
where $V_i$ $1\leq i\leq l$ are $U$-invariant mutually orthogonal eigenspaces of dimension $1$, if $1\leq i\leq s$ the corresponding eigenvalue is $1$, and for $s\leq i\leq l$ the common eigenvalue is $-1$; moreover $n-l$ is even and the subspaces $V_{l+1}, \ldots , V_{l+k}$ are $2$-dimensional $U$-invariant subspaces such that all of them are orthogonal to the $1$-dimensional ones. The restriction of $U$ to a $2$-dimensional component is a \emph{generalized rotation} with respect to an Auerbach basis $\{a_s,b_s\}$ defined by the matrix
$$
\left[A|_{\lin\{a_s,b_s\}}\right]_{\{a_s,b_s\}}=\left(\begin{array}{cc}
\cos\varphi & \sin\varphi \\
-\sin\varphi & \cos\varphi
\end{array}
\right) \mbox{ where } 0< \varphi \leq 2\pi
$$
\end{theorem}

\begin{proof} Since $V$ is an orthogonal direct sum of the eigenspaces of $U$ we have $n$ mutually orthogonal eigenvectors of $U$, say $u_1,\ldots, u_n$. Since $X$ is a finite dimensional real Banach space the eigenvalues $\lambda_1,\ldots ,\lambda_l$ corresponding to $u_1,\ldots, u_l$ are real numbers and the remaining eigenvalues $\lambda_{l+1},\ldots, \lambda_{n}$ are complex ones.

First examine the eigenvalues $\lambda_1,\ldots ,\lambda_l$. Since $U$ is an isometry we have only two possibilities for their values, these are $1$ and $-1$. We can assume that $\lambda_1=\cdots =\lambda_s=1$ and  $\lambda_{s+1}=\cdots =\lambda_l=-1$. In the subspace generated by the first $s$ eigenvectors every vector is an eigenvector with eigenvalue $1$ thus we can choose $u_1,\ldots ,u_s$ as the elements of an Auerbach basis (hence they are mutually orthogonal vectors). We choose the basis $\{u_{s+1},\ldots,u_{l}\}$ analogously from the eigenspace of eigenvalue $-1$. Since two eigenvectors corresponding to distinct eigenvalues are mutually orthogonal to each other, we get the orthogonality property of the statement about the first $l$ eigenspaces.

Assume now that $\lambda_{l+(2r-1)}=\overline{\lambda}_{l+2r}$ holds for $r=1,\ldots, (n-l)/2$. Consider again the vectors $u_{l+(2r-1)}=a_{l+(2r-1)}+b_{l+(2r-1)}i$ and scalars $\lambda_{l+(2r-1)}=\alpha_{l+(2r-1)}+\beta_{l+(2r-1)}i$ such that $U(u_{l+(2r-1)})=\lambda_{l+(2r-1)}u_{l+(2r-1)}$. (See the analogous construction in the proof of Theorem 8 on adjoint abelian operators.) The real subspaces $\lin\{a_{l+(2r-1)},b_{l+(2r-1)}\}$ are invariant with respect to $U$ and have dimension 2. Since $\lambda_{l+2r}=\alpha_{l+(2r-1)}-\beta_{l+(2r-1)}i$ and $u_{l+2r}=a_{l+2r}+b_{l+2r}i=a_{l+(2r-1)}-b_{l+(2r-1)}i$ we also have that $\lin\{a_{l+2r},b_{l+2r}\}=\lin\{a_{l+(2r-1)},b_{l+(2r-1)}\}$. Hence $V_{l+(2r-1)}=V_{l+2r}=\lin\{a_{l+(2r-1)},b_{l+(2r-1)}\}$ is an eigenspace of dimension at most $2$. The case when $b_{l+(2r-1)}=\alpha a_{l+(2r-1)}$ with real $\alpha$ implies that $a_{l+(2r-1)} $ is a real eigenvector with complex eigenvalue $\lambda_i$, which is impossible thus we get the decomposition of the statement.
Since the equality $[a_{l+(2r-1)}+b_{l+(2r-1)}i,u_r]=0$ implies the respective equalities $[a_{l+(2r-1)},u_r]=0$ and $[b_{l+(2r-1)},u_r]=0$, the last statement on orthogonality is also true.
Finally from the $U$-invariant property it follows that $U$ restricted to a $2$-dimensional invariant subspace is a generalized dilatation (see Theorem 8). On the other hand $U$ is an isometry thus $|\lambda_{l+(2r-1)}|=1$ for all $r$ hence it is a general rotation as we stated.
\end{proof}

\begin{remark} We note that there are non-diagonalizable general rotations which are also isometries. In an $l_p$ space of dimension $2$ for the general rotation $F_{\pi/2}$ we get $F_{\pi/2}(x_1a_s+y_1b_s)=(y_1a_s-x_1b_s)$ and $F_{\pi/2}(x_2a_s+y_2b_s)=(y_2a_s-x_2b_s)$ showing that
$$
[F_{\pi/2}z,F_{\pi/2}v]=
$$
$$
=\frac{1}{\left(|y_2|^p+|-x_2|^p\right)^{\frac{p-2}{p}}}\left(y_1|y_2|^{p-1}\mathrm{sgn }(y_2)-x_1|-x_2|^{p-1}\mathrm{sgn }(-x_2)\right)=
$$
$$
=\frac{1}{\left(|x_2|^p+|y_2|^p\right)^{\frac{p-2}{p}}}\left(y_1|y_2|^{p-1}\mathrm{sgn }(y_2)+x_1|x_2|^{p-1}\mathrm{sgn }(x_2)\right)=[z,v].
$$
\end{remark}

\section{Geometric algebra of a Minkowski geometry}

\subsection{The group of isometries}

In geometric algebra, one studies the properties of certain algebraic entities that can be directly linked with geometric objects, and analyses how their (algebraic) properties relate to geometric properties of the underlying geometry under investigation. This approach will be applied here to the study of "strictly convex" (normed or) Minkowski spaces, that is, metric spaces whose unit balls are centrally symmetric and strictly convex bodies. Such planes have been studied for many years; see \cite{busemann} \cite{martini-swanepoel 1} \cite{martini-swanepoel 2} \cite{thompson}, and it is particulary interesting to characterize their groups of isometries or related transformation groups.  Although the lines of strictly convex non-Euclidean Minkowski planes are just their affine lines, the group of their isometries is small. Namely, it is the semi-direct product of the translation group by a finite group of even order which either consists of Euclidean rotations or is the dihedral group. This nice fact was proven
by several authors (see in \cite{garcia},\cite{thompson} and \cite{martini-spirova-strambach}).

\begin{theorem}[\cite{garcia},\cite{thompson},\cite{martini-spirova-strambach}]\label{igthmplane} If $(V, \|\cdot\|)$ is a Minkowski plane that is non-Euclidean, then the group $\mathcal{I}(2)$ of isometries of $(V, \|\cdot\|)$ is isomorphic to the semi-direct product of the translation group $\mathcal{T}(2)$ of $\mathbb{R}^2$ with a finite group of even order that is either a cyclic group of rotations or a dihedral group.
\end{theorem}

In higher dimension it is possible for the group of linear isometries to be infinite without the space being Euclidean (e.g. if the unit ball is a elliptic cylinder in $\mathbb{R}^3$). The proof can be found in \cite{thompson} uses the concept of L\"owner-John's ellipsoids. John's (L\"owner's) ellipsoid of the unit ball $C$ is the unique ellipsoid with maximal (minimal) volume contained (circumscribed) in (about) it. It is clear that every isometry which leaves the unit ball invariant, also sends these ellipsoids into themselves, respectively. A nice consequence of this fact (proved first by Auerbach in \cite{auerbach}) is the following:

\begin{corollary}[\cite{thompson},\cite{auerbach}]
If the isometry group of a Minkowski space is transitive on the unit ball of the space then the unit ball is an ellipsoid and the space is Euclidean.
\end{corollary}

On the other hand, Gruber in \cite{gruber} shows that for "most" cases the group of isometries is finite. This follows from the fact that in "most" cases a Minkowski unit ball meets the boundary of the L\"owner ellipsoid in $d(d+1)/2$ pairs of symmetric points (see in \cite{gruber}.) Using again the concept of John's ellipsoid we can prove a similar results which is also a generalization of Theorem \ref{igthmplane}.

\begin{theorem}\label{igtne}
If the unit ball $C$ of $(V, \|\cdot\|)$ does not intersect a two-plane in an ellipse, then the group $\mathcal{I}(3)$ of isometries of $(V, \|\cdot\|)$ is isomorphic to the semi-direct product of the translation group $\mathcal{T}(3)$ of $\mathbb{R}^3$ with a finite subgroup of the group of linear transformations with determinant $\pm 1$.
\end{theorem}
\begin{proof}
Since at any point of $V$ there exists a point reflection that is an isometry of $(V, \|\cdot\|)$, the group $\mathcal{I}(n)$ contains the semi-direct product of $\mathcal{T}(n)$ with a point reflection. Since $\mathcal{I}(n)$ is a closed subgroup of the Lie group of the affinities, the translation group $\mathcal{T}(n)$ is a normal subgroup of $\mathcal{I}(n)$ and $\mathcal{I}(n)$ is a semi-direct product of $\mathcal{T}(n)$ with the stabilizer $\mathcal{I}(n)_0$ of the point $0$ in $\mathcal{I}(n)$ leaving the unit ball $C$ invariant. On the other hand every isometry of $V$ is also an affine isometry thus the elements of $\mathcal{I}(n)_0$ are in the special linear group of order $n$, too (see \cite{garcia}).

For $n=3$, from Theorem \ref{nfi} we get that an isometry has at least one eigenvector and we have two possibilities, either it is a diagonalizable operator or it is not. In the second case it has a minimal invariant subspace of dimension 2. Let $\mathcal{I}_x$ be the subgroup of  $\mathcal{I}(3)_0$ containing those isometries which fix the $1$-dimensional subspace of $x$. Then the $2$-dimensional subspace orthogonal to $x$ is also invariant with respect to the elements of $\mathcal{I}_x$ (see Theorem \ref{nfi}). By Theorem \ref{igthmplane} the group $\mathcal{I}_x$ is a finite group of even order, that is a cyclic group or a dihedral group.

Consider now the John's ellipsoid $E$ (\cite{thompson}) of the unit sphere $C$. The concept of John's ellipsoid is affine invariant hence without loss of generality we can assume that $E$ is a ball inscribed into the suitable affine copy of $C$ (which for simplicity we also denote by $C$). (Now the investigated isometries are elements of $O(3)$.) Consider the group $G$ of elements of $\mathcal{I}(3)_0$ belonging to $SO(3)$. Taking into consideration that the "determinant" map $\det: \mathcal{I}(3)_0\rightarrow \{\pm 1\}$ is a surjective group homomorphism whose kernel $G$ has index $2$ in $\mathcal{I}(3)_0$, so that, $G$ is finite if and only if $\mathcal{I}(3)_0$ is so. Let a point $x$ be a common point of the boundary of $C$ and the boundary of $E$. (Of course such a point exists.) Let denote by $C^+$ the closed half sphere containing $x$ and bounded by the hyperplane orthogonal to $x$ through the origin.  If the group $G$ is infinite then the orbit of $x$ also contains infinitely many distinct points of form $T_i(x)\in \mathrm{ bd }E\cap \mathrm{ bd }C^+$ where $T_i\in \mathcal{I}(3)_0$. Since $\mathrm{ bd }E\cap \mathrm{ bd }C^+$ is compact for every $k\in \mathbb{N}$ there  are two indices $i\neq j$ such that $\|T_i(x)-T_j(x)\|\leq 1/k$ implying that $\|T_j^{-1}T_i(x)-x\|\leq 1/k$. Consider the isometry $T_j^{-1}T_i\in SO(3)$. Hence $T_j^{-1}T_i$ is rotation about an axis, say $x_k$.
Thus the points $\left(T_j^{-1}T_i\right)^l(x)$ for $l\in \mathbb{N}$ are on a two dimensional intersection of $\mathrm{ bd }E$, so they are also on a circle $E_k$. This circle through the point $x$ contains a set of points of $\mathrm{ bd }C$ with successive distance at most $1/k$ forming an $1/k$-net on it. Let denote by $y_k$ the unit normal vector of the plane of $E_k$ directed by $C^{+}$. The set $Y:=\{y_k \quad k\in \mathbb{N}\}$ is infinite and hence it has a convergent subsequence $\left(y_{k_i}\right)$ with limit $y$. Consider now the circle $E(x,y)$ defined by the intersection of $E$ with the plane through $x$ and orthogonal to $y$. It has the property that if $z\in E(x,y)$ then for every $\varepsilon>0$ there is a point $u$ of $\mathrm{ bd }E\cap \mathrm{ bd }C$ such that $\|z-u\|\leq \varepsilon$. This implies that $E(x,y)\subset \mathrm{ bd }E\cap \mathrm{ bd }C$ giving a contradiction with our assumption. Thus the group $\mathcal{I}(n)_0$ is finite and the statement is true.
\end{proof}
\begin{remark}
We note that we proved the finiteness of the point group with a stronger assumption than that of the "totally non-Euclidean" property. A method using L\"owner-John ellipsoids can not be applied to prove a more general statement in this direction because there are Minkowski spaces which are not totally non-Euclidean but the intersection of the John's ellipsoid of its unit sphere contains an ellipse. For a simple example, consider a unit ball $B$ and one of its great circles $S$. Let $H(2n,\varepsilon)$ be a regular polygon circumscribed to $(1+\varepsilon)S$ with $2n$ vertices. Now define the unit ball $C(n,\varepsilon):=\mathrm{ conv }\{B\cup H(2n,\varepsilon)\}$. It is clear that the Minkowski space with unit ball $C(n,\varepsilon)$ is not totally non-Euclidean however for small $\varepsilon$ and for large $n$ the John's ellipsoid of $C(n,\varepsilon)$ is $B$, hence $\mathrm{ bd }C(n,\varepsilon)\cap \mathrm{ bd }B$ contains a circle.
\end{remark}
This motivates the following problem:
\begin{problem}
Is it true or not that if for $n\geq 3$ the Minkowski $n$-space is totally non-Euclidean one (see Definition 4) then its isometry group $\mathcal{I}(n)$ is a semi-direct product of the translation group $\mathcal{T}(n)$ with a finite subgroup of $SL(n)$?
\end{problem}

\subsection{Affine reflections, left reflections}

On an affine plane an axial affinity of order two called by \emph{affine reflection}.
Any affine reflection $\alpha$ has an \emph{axis} $G$ and leaves any line of precisely one parallel class $\mathcal{A}$ not containing $G$ invariant. We call this parallel class \emph{the direction} of $\alpha$. We say that a line has a direction $\mathcal{A}$ if it belongs to $\mathcal{A}$.
We collect the known statements on affine reflections in a strictly convex Minkowski plane. This results can be found in \cite{martini-spirova-strambach}.

\begin{itemize}
\item Let $\alpha$ be the reflection at the point $0$ with the stabilizer $\mathcal{I}(2)_0$ of the isometry group $\mathcal{I}(2)$ of a strictly convex Minkowski plane. If $\mathcal{I}(2)_0$ is the dihedral group, then the lines $G_{\tau}$ and $G_{\tau\alpha}$, which are axes of the affine reflections $\tau$ resp. $\tau\alpha$, are mutually orthogonal.

\item Let $\Psi_{G_1}$ and $\Psi_{G_2}$ be two affine reflections in the non-parallel lines $G_1$ and $G_2$.

(a) The affine reflections $\Psi_{G_1}$ and $\Psi_{G_2}$ leave a line $H$ different to $G_1$ and $G_2$ and passing through $G_1\cap G_2$ invariant if and only if their product is a shear with $H$ as axis.

(b) The product $\Psi_{G_1}\circ \Psi_{G_2}$ has only the intersection $p$ of $G_1$ and $G_2$ as a fixed point if and only if $\Psi_{G_1}\circ \Psi_{G_2}$ is not a shear.
\end{itemize}

The scarcity of isometries of non-Euclidean Minkowski planes motivated H. Martini and M. Spirova to introduce left reflections as a conceptual tool for investigating strictly convex Minkowski planes (see \cite{martini-spirova}).
\begin{defi} Given a line $G$ in a strictly convex Minkowski plane $(V,\|\cdot \|)$, we define a transformation
$$
\Psi_G: V\rightarrow V \quad \Psi_G(p)=p'
$$
to be a left reflection about the line $G$ if

 (i) $p'= p$ holds for all $p\in G$,

 (ii) $p'\neq p$ and $\langle p,p'\rangle\bot_B G$ hold for all points $p\in V\setminus G$, and

 (iii) the midpoint of the segment $[p, p']$ lies on $G$ for all points $p\in V$.
\end{defi}

In general, left reflections are not isometries, but they are strongly related to the notion of Birkhoff orthogonality: if every left reflection in a strictly convex Minkowski plane preserves Birkhoff orthogonality, then this plane is a Radon plane. The advantage of left reflections is that there exists one for any line. In general, they generate the special or equi-affine group of the real affine plane, i.e., the group of affine transformations of determinant 1 (see \cite{veblen}). By imposing specific conditions on left reflections, their products, or the group generated by them, one can characterize the Euclidean plane as well as Radon planes in new ways. For example, a strictly convex Minkowski plane is Euclidean if every left reflection maps circles into circles or preserves James orthogonality, or if the product of any two distinct left reflections is an isometry. In particular, Bachmann's approach to geometry (see \cite{bachmann}) can be used in an efficient way. It were shown that smooth, strictly convex Minkowski or strictly convex Radon planes are Euclidean if the Three-Reflections Theorem holds for left reflections. The left reflections in any strictly convex Minkowski plane are affine reflections. In the Euclidean case the left reflections generate already the full equi-affine group. It can be proved that the set of all left reflections of the Minkowski plane generates a proper subgroup of the equi-affine group then the plane is Euclidean and the subgroup is the full equi-affine group. Using products of two left reflections, one can characterize the singularities of the unit circle in a strictly convex Minkowski plane. Moreover, it is clarified when the product of a left reflection in a line $G$ with the reflection in a point incident with $G$ is again a left reflection.
We noted that left reflections, as defined here, are affine reflections. It was Martini's and Spirova's aim to investigate how the concept of left reflections fits into the geometry of (special) normed planes. They started with the following observation.

For any $y\neq 0$  in a strictly convex Minkowski plane there exists a unique direction $x$ with $x \bot_B y$. For any $x\neq  0$ in a smooth Minkowski plane, there exists a unique direction $y$ such that $x \bot_B y$ (cf. \cite{james 1}, \cite{james 2}). Clearly, in affine geometry the notion of direction is given by a class of parallel lines represented by a respective vector.

We notice also that all results obtained for this type of left reflections can be analogously derived for the correspondingly defined concept of right reflections in lines for smooth Minkowski planes. The observation above guarantees that also this type of transformation is well defined, and that for smooth Minkowski planes, we may replace in (ii) the condition $<p,p'>\bot_B G$ by the condition $G\bot_B <p,p'>$, and so all results derived for left reflections in strictly convex planes hold analogously for right reflections in smooth planes; see \cite{martini-spirova}. Moreover, in strictly convex as well as smooth Minkowski planes any left reflection is also a right-reflection if and only if $(V, \|\cdot\|)$ is a Radon plane; see [\cite{martini-swanepoel 1}, Section 6.1.2]).

We now give a collection of important results on left reflections from \cite{martini-spirova-strambach} and \cite{martini-spirova}.

\begin{itemize}

\item Let $\Psi_G$ be a left reflection in a strictly convex Minkowski plane $(V, \|\cdot\|)$. Then $\Psi_G$ is involutory and is affine. Moreover it has the following properties.

(i) If $H$ is an arbitrary line and $H^{\Psi_G} = H'$, then either $H \| H' \| G$ or $ H\cap H'\in G$.

(ii) The only fixed lines of $\Psi_G$ except for $G$ are those that are Birkhoff orthogonal to $G$, i.e., the lines Birkhoff orthogonal of $G$ form the direction of $\Psi_G$.

\item The product of two left reflections in parallel lines of a strictly convex Minkowski plane is a translation.

\item  The product of three left reflections in parallel lines in a strictly convex Minkowski plane $(V, \|\cdot\|)$ is a left reflection in another line belonging to the same pencil of parallel lines.

\item Let $G_1$ and $G_2$ be two (not necessarily distinct) parallel lines in a strictly convex Minkowski plane $(V, \|\cdot\|)$. Let $p$ be an arbitrary non-zero vector of $(V, \|\cdot\|)$ and $\Theta_p$ be the translation through $p$. If $\Psi_{G_1}$ and $\Psi_{G_2}$ are
the reflections in $G_1$ and $G_2$, then $\Psi_{G_2} \circ  \Theta_p \circ  \Psi_{G_1}$ is a translation.

\item If every left reflection in a strictly convex, smooth Minkowski plane preserves Birkhoff orthogonality, then it is a Radon plane.

\item A strictly convex and smooth Minkowski plane $(V, \|\cdot\|)$ is Euclidean if, and only if, for arbitrary three distinct lines $G_1$,$G_2$, and $G_3$ having a common point $p$ there exists a line $G_4$ passing through $p$ such that
$\Psi_{G_3}\circ \Psi_{G_2} \circ \Psi_{G_1} = \Psi_{G_4}$. (Here the assumption "smooth" can be replaced by the assumption "Radon plane" (Th. 4.4 in \cite{martini-spirova-strambach}).

\item A strictly convex Minkowski plane is Euclidean if and only if the left reflections generate a proper closed subgroup of the equiaffine group.

\item A strictly convex Minkowski plane is Euclidean if and only if the image of any circle with respect to any left reflection is also a circle.

\item If every left reflection in a strictly convex Minkowski plane preserves James orthogonality, then it is a Euclidean plane.

\item A strictly convex Minkowski plane that is smooth or a Radon plane is Euclidean if and only if for any two intersecting lines $G_1$ and $G_2$ and an arbitrary line $G_1'$ passing through $\{p\}=G_1\cap G_2$ there exists a line $G_2'$ through $p$ such that $\Phi_{G_2',G_1'}=\Phi_{G_2,G_1}$.

\item In a strictly convex Minkowski plane every product of two left reflections is an isometry if and only if the plane is Euclidean.
\end{itemize}

We note that the concept of left (affine) reflections without any hardness can be extracted for higher dimensions and a possible research problem is to investigate statements analogous to the mentioned above. For example the definition of a left reflection in a strictly convex Minkowski $n$-space could be the following:
\begin{defi} Given a hyperplane $G$ in a strictly convex Minkowski space $(V,\|\cdot \|)$, we define a transformation
$$
\Psi_G:V\rightarrow V \quad  \Psi_G(p)=p'
$$
to be a left reflection in the hyperplane $G$ if

 (i) $p'= p$ holds for all $p\in G$,

 (ii) $p'\neq p$ and $\langle p,p'\rangle\bot_B G$ hold for all points $p\in V\setminus G$, and

 (iii) the midpoint of the segment $[p, p']$ lies on $G$ for all points $p\in V$.
\end{defi}
From the definition it is obvious that a left reflection is an involutory mapping. On the other hand, by the definition of Minkowski distance it is also clear that it sends a $k$-dimensional subspace into a $k$-dimensional one and ideal points into ideal ones, respectively. This means that it is an affinity of the space. Hence if the origin is on the hyperplane $G$ then it is also a linear mapping.

\end{document}